\documentclass{amsart}
\usepackage{amssymb,amsmath,amsfonts,amsthm,epsfig,amscd}
\usepackage{stmaryrd}
\usepackage[all,cmtip,poly]{xy}
\usepackage{svn}
\usepackage{mathtools}

\newtheorem{thm}[subsection]{Theorem}
\newtheorem{lemma}[subsection]{Lemma}
\newtheorem{lem}[subsection]{Lemma}

\newtheorem{cor}[subsection]{Corollary}

\newtheorem{prop}[subsection]{Proposition}

\theoremstyle{definition}

\theoremstyle{remark}
\newtheorem{remark}[subsection]{Remark}
\newtheorem{rem}[subsection]{Remark}

\def\numequation{\addtocounter{subsubsection}{1}\begin{equation}}
\def\nummultline{\addtocounter{subsubsubsection}{1}\begin{multline}}
\def\anumequation{\addtocounter{subsection}{1}\begin{equation}}

\newif\iffinalrun
\iffinalrun
\else
 \fi

\iffinalrun
  \newcommand{\need}[1]{}
  \newcommand{\mar}[1]{}
\else
  \newcommand{\need}[1]{{\tiny *** #1}}
  \newcommand{\mar}[1]{\marginpar{\raggedright\tiny #1}}
\fi

\newcommand{\C}{\CC}

\newcommand{\Q}{\QQ}
\newcommand{\R}{\RR}
\newcommand{\Z}{\ZZ}

\newcommand{\CC}{{\mathbb C}}

\newcommand{\QQ}{{\mathbb Q}}
\newcommand{\RR}{{\mathbb R}}

\newcommand{\ZZ}{{\mathbb Z}}

\renewcommand{\bf}{\ensuremath{\mathbf{f}}}

\newcommand{\cI}{{\mathcal I}}

\newcommand{\cM}{{\mathcal M}}

\newcommand{\cO}{{\mathcal O}}

\newcommand{\cX}{{\mathcal X}}

\DeclareMathOperator{\End}{End}

\DeclareMathOperator{\GL}{GL}

\DeclareMathOperator{\Spec}{Spec}

\newcommand{\into}{\hookrightarrow}
\newcommand{\onto}{\twoheadrightarrow}

\DeclarePairedDelimiter\floor{\lfloor}{\rfloor}

\begin{document}
\title{On the image of complex conjugation in certain Galois representations}

\author[A. Caraiani]{Ana Caraiani}\email{caraiani@princeton.edu}
\address{Department of Mathematics, Princeton University, Fine Hall,
Washington Rd., Princeton, NJ 08544, USA}
\author[B. V. Le Hung]{Bao V. Le Hung}\email{lhvietbao@googlemail.com} 
\address{Department of Mathematics,
University of Chicago,
5734 S. University Avenue,
Chicago, IL 60637, USA}
\thanks{A.C.\ was partially
  supported by the NSF Postdoctoral Fellowship DMS-1204465.}
\maketitle

\begin{abstract}
We compute the image of any choice of complex conjugation on the Galois representations associated to regular algebraic cuspidal automorphic representations and to torsion classes in the cohomology of locally symmetric spaces for $GL_n$ over a totally real field $F$. 
\end{abstract}

\section{Introduction}

The goal of this note is to describe the image of any choice of complex conjugation on the Galois representations associated to regular algebraic cuspidal automorphic representations and to (mod $p$) torsion classes in the cohomology of locally symmetric spaces for $GL_n$ over a totally real field $F$. Since any choice of complex conjugation has eigenvalues $\pm 1$, the key computation is to determine how many $+1$'s and how many $-1$'s occur; we do this by showing that their numbers differ by at most $1$. Our results are conditional on Arthur's work~\cite{arthur}.

In the case of regular algebraic cuspidal automorphic representations of $GL_n(\mathbb{A}_F)$ which are essentially self-dual this is known in almost all cases, due to Taylor~\cite{taylorcc} (when $n$ is odd and under the assumption that the corresponding Galois representation is irreducible) and Taibi~\cite{taibi} (all cases when $n$ is odd and most cases when $n$ is even). We note that in the essentially self-dual case when $n$ is odd, the corresponding Galois representation occurs in the \'etale cohomology of a certain Shimura variety. Taylor makes use of a geometric realization of complex conjugation and studies its action on the Hodge filtration of the Betti cohomology of this Shimura variety. Taibi uses $p$-adic interpolation techniques (eigenvarieties) to extend Taylor's result to almost all essentially self-dual cases. 

Recently, Harris, Lan, Taylor and Thorne used more geometric $p$-adic interpolation techniques in~\cite{HLTT} to construct Galois representations associated to regular algebraic cuspidal automorphic representations of $GL_n(\mathbb{A}_F)$ which do not need to be essentially self-dual. Later, Scholze gave a different construction in ~\cite{scholzetorsion}, still via $p$-adic interpolation, which also applies to torsion classes in the cohomology of the corresponding locally symmetric space.  

In this paper, we extend the result concerning the image of complex conjugation beyond the essentially self-dual case using the very techniques which led to the construction of the Galois representations we are interested in. We follow Scholze's approach rather than that of~\cite{HLTT}. Just as the construction of Galois representations for torsion classes in the case when $F$ is totally real, our result makes use of the transfer of a cusp form on $Sp_{2n}$ to $GL_{2n+1}$ and is therefore dependent on~\cite{arthur}, which is still conditional on the stabilization of the twisted trace formula.

Let $F$ be a totally real field and let $\pi$ be a cuspidal automorphic representation of $GL_n(\mathbb{A}_F)$ such that $\pi_\infty$ is regular $L$-algebraic. Let $S$ be a finite set of places of $F$, which contains all the places where $\pi$ is ramified, and let $G_{F,S}$ denote the Galois group of the maximal extension of $F$ unramified outside $S$. Then there exists a Galois representation \[\sigma_\pi: G_{F,S} \to GL_n(\bar{\mathbb {Q}}_p)\] which satisfies local-global compatibility at all finite places $v\not\in S$. More precisely, for every finite place $v\not \in S$, the Satake parameters of $\pi_v$ are the same as the eigenvalues of $\sigma_\pi(\mathrm{Frob}_v)$ (see, for example, Corollary V.4.2 of~\cite{scholzetorsion}). We prove the following

\begin{thm}\label{thm: cuspidal automorphic}
Let $\pi$ be a regular $L$-algebraic, cuspidal automorphic representation of $\GL_n(\mathbb{A}_F)$, with associated (p-adic) Galois representation $\sigma_\pi$. Let $c\in \mathrm{Gal}(\bar{F}/F)$ be a choice of complex conjugation. Then $\mathrm{tr}(\sigma_\pi)(c)=0$ if $n$ is even and $\mathrm{tr}(\sigma_\pi)(c)=\pm 1$ if $n$ is odd. 
\end{thm}

\begin{remark}
This result can be regarded as part of local-global compatibility at the Archimedean primes $v$ of $F$: The Langlands parameter $\phi_v: W_{\R}=\C^\times\cup \C^\times j\to \mathrm{GL}_n(\C)$ of $\pi_v$ is expected to contain the data of the Hodge-Tate weights of $\sigma_\pi$ (via the weights of $\C^\times \subset W_{\R}$), while the image of complex conjugation $c_v$ should be controlled by the conjugacy class of $\phi_v(j)$. Under the regularity assumption, there are up to sign at most two choices of $\phi_v(j)$, corresponding exactly to the outcome stated in Theorem \ref{thm: cuspidal automorphic}. We also remark that in this setting, the Hodge-Tate part of the compatiblity will be a consequence of~\cite{Ila}.
\end{remark}
We prove the theorem by proving the same result for the Galois representations (or more precisely, determinants) associated to systems of Hecke eigenvalues occurring in the cohomology of locally symmetric spaces for $GL_n/F$. For a sufficiently small level $K\subset GL_n(\mathbb{A}_{F,f})$, define the locally symmetric space \[ X_K: = GL_n(F)\backslash\left(GL_n(\mathbb{A}_{F,f})/K \times GL_n(F\otimes_{\mathbb{Q}}\mathbb R)/ \mathbb R_{>0}K_\infty \right),\] where $K_\infty\subset  GL_n(F\otimes_{\mathbb{Q}}\mathbb R)$ is a maximal compact subgroup. Since $\pi$ is regular $L$-algebraic, $\pi':=\pi|\ \cdot|^{(n+1)/2}$ is regular $C$-algebraic, i.e. cohomological. Therefore, some twist of $\pi'$ by a character of order $2$ occurs in $H^i(X_K,\cM_{\xi,K})\otimes_{\mathbb {\bar{Z}}_p}\mathbb{C}$, for some algebraic representation $\xi$ of $\mathrm{Res}_{F/\mathbb{Q}}GL_n$ over $\mathbb{C}$, some integer $i$ and some sufficiently small level $K$.

Let \[\mathbb{T}_{F,S}:=\bigotimes_{v\not\in S}\mathbb{T}_v, \mathbb{T}_v=\mathbb{Z}_p[GL_n(F_v)//GL_n(\cO_{F_v})]\] be the abstract Hecke algebra. The representation $\pi'$ determines a homomorphism $\psi:\mathbb{T}_{F,S}\to \mathbb{\bar Z}_p$ (a \textit{system of Hecke eigenvalues}), which factors through some \[\mathrm{Im}(\mathbb{T}_{F,S}\to \End_{\mathbb{\bar Z}_p}(H^i(X_K,\cM_{\xi,K}))). \] (A priori, $\pi'$ determines a homomorphism of the Hecke algebra into $\mathbb{\bar Q}_p$, but local-global compatibility at places $v\not\in S$ guarantees that this actually lands inside $\mathbb{\bar Z}_p$). Let $\bar \psi: \mathbb{T}_{F,S} \to \mathbb{\bar{F}}_p$ be obtained from $\psi$ by composing with the natural map $\mathbb{\bar{Z}}_p\to \mathbb{\bar F}_p$. Since $\pi'$ occurs in $H^i(X_K, \cM_{\xi, K})\otimes_{\mathbb{\bar Z}_p}\mathbb{\bar Q}_p$, we see by Proposition 1.2.3 of ~\cite{ashstevens} that \[H^i(X_K, \cM_{\xi, K}\otimes_{\mathbb{\bar Z}_p}\mathbb{\bar F}_p)[\bar \psi]\not = 0.\]An argument using the Hochschild-Serre spectrial sequence (see, for example, the proof of Theorem V.4.1 of ~\cite{scholzetorsion}) tell us that \[H^i(X_K, \mathbb{\bar F}_p)[\bar \psi]\not = 0,\] where we have possibly replaced $K$ by a smaller compact open subgroup. Thus, the reduction mod $p$ of the system of Hecke eigenvalues corresponding to $\pi'$ occurs in the mod $p$ cohomology of $X_K$. There is also a mod $p^n$ version of the above picture.

\begin{thm}\label{thm: torsion}
Let $\psi$ be a system of Hecke eigenvalues occurring in $H^i(X_K, \mathbb{\bar F}_p)$ and let $\sigma_\psi$ be the corresponding Galois representation. Let $c\in \mathrm{Gal}(\bar{F}/F)$ be a choice of complex conjugation. Then $\sigma_\pi(c)$ has $+1$ as an eigenvalue with multiplicity $\lceil \frac{n-1}{2} \rceil$ and $-1$ as an eigenvalue with multiplicity $\lfloor \frac{n+1}{2} \rfloor $.
\end{thm}

The Galois representation $\sigma_\psi$ is obtained by specializing an $n$-dimensional continuous determinant, which is extracted from a $2n+1$-dimensional determinant which in turn interpolates the Galois representations associated to regular algebraic automorphic representations of $G_0:= \mathrm{Sp}_{2n}/F$. This essentially shows that $1\oplus \sigma_\psi\oplus \check{\sigma}_\psi$ is congruent to a Galois representation associated to (the transfer to $GL_{2n+1}$ of) a cuspidal automorphic representation of $G_0$. We then compute the characteristic polynomial of any choice of complex conjugation on the latter Galois representation (using Taibi's main theorem) and therefore determine the characteristic polynomial of any complex conjugation on the former. This gives Theorem~\ref{thm: torsion}. Adapting this for mod $p^n$ systems of Hecke eigenvalues then gives us Theorem~\ref{thm: cuspidal automorphic}.

\begin{rem}
Theorem \ref{thm: cuspidal automorphic} applies in particular to the essentially self-dual representations not covered by Taibi's theorem. However, our proof does not give a new proof of this result, as it was one of the inputs of our argument.  
\end{rem}

In practice, some technical complications arise. Theorem V.4.1 of ~\cite{scholzetorsion} guarantees that there is a determinant valued in the quotient of $\mathbb{T}_{F,S}$ which acts faithfully on $H^i(X_K, \mathbb{\bar F}_p)$, glued out of determinants valued in similar quotients acting on the interior cohomology $H^i_!(X_K, \mathbb{\bar F}_p)$ and on the cohomology of the boundary of the Borel-Serre compactification of $X_K$, $H^{i+1}(X^{\mathrm{BS}}_K, \mathbb{\bar F}_p)$. (We remark that, if we were merely interested in those $\bar \psi$ which are reductions of characteristic $0$ systems of Hecke eigenvalues corresponding to cuspidal automorphic representations, we could use the Hochschild-Serre spectral sequence before reducing to $\mathbb{\bar F}_p$ to ensure that the system of Hecke eigenvalues mod $p$ occurs in interior cohomology with trivial coefficients.) 

The determinant valued in the Hecke algebra acting on $H^i_!(X_K, \mathbb{\bar F}_p)$ is constructed by showing that the interior cohomology above contributes to the cohomology of the boundary of the Borel-Serre compactification of the locally symmetric space for $G:= \mathrm{Res}_{F/\mathbb{Q}}G_0$. The torsion cohomology of the locally symmetric space for $G$ is then related to classical cusp forms. Only this part is directly related to \textit{cuspidal} automorphic forms on $G_0$. We review the construction of this determinant in Section~\ref{completed cohomology}.

On the other hand, the determinant valued in the Hecke algebra acting on $H^{i+1}(X^{\mathrm{BS}}_K, \mathbb{\bar F}_p)$ is glued together out of the determinant for interior cohomology and determinants for locally symmetric spaces for $GL_n'$ with $n'<n$, whose cohomology contributes to the boundary cohomology of the Borel-Serre compactification of $X_K$. This allows an induction argument. We review the geometry of the boundary of the Borel-Serre compactification and explain the construction of the determinant in Section~\ref{boundary}. In section~\ref{complex conjugation}, we put all of this together to compute the characteristic polynomial of complex conjugation. 

\subsection{Acknowledgements}
We thank Frank Calegari for several remarks leading up to the material in Section~\ref{completed cohomology} and for comments on an earlier draft of this paper. We thank Toby Gee for raising the question of studying the image of complex conjugation to one of us and for comments on an earlier draft of this paper. We thank Sophie Morel for explaining the basics of compactifications of locally symmetric spaces. Part of this work was done while both of us were at the Institute for Advanced Studies and Princeton University, and we would like to thank both institutions for their hospitality. This material is also based upon work supported by the National Science Foundation under Grant No.0932078000 while the authors were in residence at the Mathematical Sciences Research Institute in Berkeley, California, during the Fall 2014 semester.
  
\section{Determinants and completed cohomology}
\label{completed cohomology}

In this section, we recall the construction due to Scholze of determinants (and, therefore, Galois representations) associated to systems of Hecke eigenvalues occurring in the completed cohomology of Shimura varieties for symplectic groups. 

Recall that $G_0:=\mathrm{Res}_{F/\mathbb{Q}}Sp_{2n}$. In this section, all Shimura varieties and Hecke algebras will be with respect to $G_0$. For $K^p\subset G_0(\mathbb A^{\infty,p})$ a sufficiently small compact open subgroup, let $\mathbb{T}_{K^p}$ be the abstract Hecke algebra over $\Z_p$ at level $K^{p,S}$, where $S$ is the set of places where $K^p$ is not hyperspecial. For $K_p\subset G_0(\mathbb{Q}_p)$ compact open, let $X_{K^pK_p}$ be the Shimura variety for $G_0$ of level $K^pK_p$ and let $X^*_{K^pK_p}$ be its minimal compactification. Let $\mathcal{O}$ be a finite extension of $\Z_p$, which we will use as our coefficients, and let $\pi$ be a uniformizer.

Let $d$ be the dimension of $X_{K^pK_p}$ and let $\mathbb{\tilde{T}}_{K^p}(m)$ be the inverse limit of the images $\mathbb{T}_{K_pK^p}(m)$ of $\mathbb{T}_{K^p}$ acting on $\oplus_{i=0}^{2d} H^i_c(X_{K^pK_p},\mathcal{O}/\pi^m)$ as $K_p$ becomes arbitrarily small. It is an inverse limit of finite discrete rings, and hence the inverse limit topology makes it a compact topological ring. Let $\tilde{H}^i_{c,K^p}(\mathcal{O}/\pi^m)=\varinjlim_{K_p} H^i_c(X_{K^pK_p},\mathcal{O}/\pi^m)$ denote the completed compactly supported cohomology with $\mathcal{O}/\pi^m$-coefficients. Let $\hat{\mathbb{T}}_{K^p}$ be the profinite completion of $\mathbb{T}_{K^p}$.Then $\mathbb{\tilde{T}}_{K^p}(m)$ is the image of $\hat{\mathbb{T}}_{K^p}$ in $\End_{\mathcal{O}/\pi^m}\left(\oplus_{i=0}^{2d}(\tilde{H}^i_{c,K^p}(\mathcal{O}/\pi^m))\right)$, and the action is continuous for the discrete topology on $\tilde{H}^i_{c,K^p}(\mathcal{O}/\pi^m)$. Note that if we use completed cohomology instead we would get the same quotient of $\hat{\mathbb{T}}_{K^p}$, by Poincar\'e duality.

Each $\mathbb{T}_{K_pK^p}(m)$ is a finite ring, and hence is the product of finitely many local rings, which are in bijection with its maximal ideals. A maximal ideal is the same data as a homomorphism $\mathbb{T}_{K_pK^p}(m)\to \overline{\mathbb{F}}_p$ (i.e. an $\bar{\mathbb{F}}_p$-system of Hecke eigenvalues), up to an automorphism of $\overline{\mathbb{F}}_p$. Each such system of eigenvalues is valued in a finite field. If $\mathfrak{m}$ is a maximal ideal of $\mathbb{T}_{K_pK^p}(m)$, then  $\oplus_{i=0}^{2d} H^i_c(X_{K^pK_p},\bar{\mathbb{{F}}}_p)_\mathfrak{m}\neq 0$, and (equivalently) $\oplus_{i=0}^{2d} H^i(X_{K^pK_p},\bar{\mathbb{{F}}}_p)[\mathfrak{m}]\neq 0$. If $H^i(X_{K^pK_p},\bar{\mathbb{{F}}}_p)[\mathfrak{m}]\neq 0$, we say that the system of Hecke eigenvalues corresponding to $\mathfrak{m}$ \textit{occurs} in $H^i_c(X_{K^pK_p},\bar{\mathbb{{F}}}_p)$. A non-zero cohomology class in this space is an eigenvector for $\mathbb{T}_{K^p}$, with the given system of Hecke eigenvalues, justifying the terminology. Observe that the maximal ideals of $\mathbb{T}_{K_pK^p}(m)$ and $\mathbb{T}_{K_pK^p}(1)$ are naturally in bijection with each other.

\begin{prop} There are finitely many systems of Hecke eigenvalues for $\mathbb{T}_{K_pK^p}(1)$ as $K_p$ varies. 
\end{prop}

\begin{proof} Define $K(m):=\left\{ g\in G_0(\mathbb{Z}_p)|g\equiv 1\pmod{p^m}\right\}$. It suffices to see that every system of Hecke eigenvalues occurring in $H^i(X_{K^pK(m+1)},\mathbb{\bar F}_p)$ also occurs in some $H^{i'}(X_{K^pK(m)},\mathbb{\bar F}_p)$ with $i'\leq i$, whenever $m\geq 1$. For this, we use the Hochschild-Serre spectral sequence:
\[E_2^{ij}=H^i(K(m+1)/K(m), H^j(X_{K^pK(m+1)},\mathbb{\bar F}_p)) \Rightarrow H^{i+j}(X_{K^pK(m)},\mathbb{\bar F}_p),\] which is  $\mathbb{T}_{K^p}$-equivariant.
First, note that $K(m+1)/K(m)$ is an abelian $p$-group, so any element of $K(m+1)/K(m)$ has $1$ as its only eigenvalue. Second, note that $\mathbb{T}_{K^p}[K(m+1)/K(m)]$ is commutative. Therefore, every system of  $\mathbb{T}_{K^p}$-eigenvalues that occurs in $ H^{j}(X_{K^pK(m+1)},\mathbb{\bar F}_p)$ also occurs in $H^0(K(m+1)/K(m),  H^{j}(X_{K^pK(m+1)},\mathbb{\bar F}_p))$. 

If the eigenvector survives in the $E_\infty$ page of the Hochschild-Serre spectral sequence, we are done. Otherwise, a diagram chase and Proposition 1.2.2 of ~\cite{ashstevens} tell us that the system of Hecke eigenvalues has to occur in some $H^{i'}(K(m+1)/K(m), H^{j'}(X_{K^pK(m+1)},\mathbb{\bar F}_p))$ with $j'<j$. But then Lemma~\ref{lem: system of Hecke eigenvalues in group cohomology} and the argument above tells us it must also occur in $H^0(K(m+1)/K(m), H^{j'}(X_{K^pK(m+1)},\mathbb{\bar F}_p))$. We then proceed as above until we possibly reach $j'=0$, in which case the system of Hecke eigenvalues occurs in $H^0(X_{K^pK(m)},\mathbb{\bar F}_p)$. 
\end{proof}

\begin{lemma}\label{lem: system of Hecke eigenvalues in group cohomology} Every system of $\mathbb{T}_{K^p}$-eigenvalues occurring in \[H^{i}(K(m+1)/K(m), H^{j}(X_{K^pK(m+1)},\mathbb{\bar F}_p))\] also occurs in $H^{j}(X_{K^pK(m+1)},\mathbb{\bar F}_p)$.
\end{lemma}
\begin{proof} Define \[\mathbb{T}_1:= \mathrm{Im}\left(\mathbb{T}_{K^p}\to \End_{\mathbb{\bar F}_p} \left(H^{j}(X_{K^pK(m+1)},\mathbb{\bar F}_p)\right) \right)\] and  \[\mathbb{T}_2:= \mathrm{Im}\left(\mathbb{T}_{K^p}\to \End_{\mathbb{\bar F}_p} \left(H^i(K(m+1)/K(m), H^{j}(X_{K^pK(m+1)},\mathbb{\bar F}_p))\right) \right).\] Then $\mathbb{T}_2$ is a quotient of $\mathbb{T}_1$, so any system of Hecke eigenvalues occurring in $H^i(K(m+1)/K(m), H^{j}(X_{K^pK(m+1)},\mathbb{\bar F}_p))$ determines a maximal ideal of $\mathbb{T}_1$. Now, it suffices to notice that, since $H^{j}(X_{K^pK(m+1)},\bar{\mathbb{F}}_p)$ is a finite-dimensional $\mathbb{\bar F}_p$-vector space, every maximal ideal of $\mathbb{T}_1$ is in the support of $H^{j}(X_{K^pK(m+1)},\mathbb{\bar F}_p)$. Using Nakayama's lemma and Proposition 1.2.2 of ~\cite{ashstevens}, we see that every maximal ideal of $\mathbb{T}_1$ determines a system of Hecke eigenvalues occurring in $H^{j}(X_{K^pK(m+1)},\mathbb{\bar F}_p)$. 
\end{proof}

\begin{cor} The ring $\mathbb{\tilde T}_{K^p}(m)$ is a product of finitely many complete profinite local rings, each with finite residue field. 
\end{cor}
\begin{proof}
$\mathbb{\tilde T}_{K^p}(m)=\varprojlim_{K_p} \mathbb{T}_{K_pK^p}(m)$, and by the previous proposition, the transition maps are eventually surjective maps between products of local Artinian rings which the same number of factors. This gives the factorization of $\mathbb{\tilde T}_{K^p}(m)$ into a product of projective limits of local finite rings.
\end{proof}

The upshot of this corollary is that we can now work after localizing at each of the finitely many maximal ideals of $\mathbb{\tilde T}_{K^p}(m)$ individually. Let $\mathfrak{m}_i$ for $i=1,\dots,N$ be the finitely many maximal ideals of  $\mathbb{\tilde T}_{K^p}(m)$. Let  $\mathbb{T}_{K^pK_p,k, \mathfrak{m}_i}$ be the image of  $\mathbb{T}_{K^p}$ acting on the finite-dimensional space of cusp forms of level $K_pK^p$ and weight $m_0j$ with $1\leq j\leq k$ and $m_0$ a sufficiently large integer:
\[\mathbb{T}_{K^pK_p,k, \mathfrak{m}_i}:= \mathrm{Im}\left(\mathbb{ T}_{K^p} \to \End_{\C_p}\left(\oplus_{j=1}^{k} H^0(\cX^*_{K^pK_p}, \omega^{\otimes m_0j}_{K^pK_p}\otimes \cI)_{\mathfrak{m}_i}\right) \right). \] (Here, $\cX^*_{K^pK_p}$ is the adic space associated to the complex algebraic variety $X^*_{K^pK_p}$, based chaged from $\C$ to $\C_p$.)

Let $\mathbb{T}^{\mathrm{cl}}_{K^p, \mathfrak{m}_i}$ be the maximal quotient of $\hat{\mathbb{T}}_{K^p}$ over which all maps $\hat{\mathbb{T}}_{K^p}\to \mathbb{T}_{K^pK_p,k, \mathfrak{m}_i}$ factor. Note that $\mathbb{T}^{\mathrm{cl}}_{K^p, \mathfrak{m}_i}$ is the image of $\hat{\mathbb{T}}_{K^p}$ inside $\varprojlim_{k,K_p}\mathbb{T}_{K^pK_p,k, \mathfrak{m}_i}$, hence is compact. Furthermore it is dense (because its composition with the projection to each $\mathbb{T}_{K^pK_p,k,\mathfrak{m}_i}$ is surjective), hence $\mathbb{T}^{cl}_{K^p, \mathfrak{m}_i}=\varprojlim_{k,K_p}\mathbb{T}_{K^pK_p,k, \mathfrak{m}_i}$.
\begin{prop}\label{prop: classical determinants}
There is a continuous $(2n+1)$-dimensional determinant $\tilde{D}$ of $G_{F,S}$ with values in $\mathbb{T}_{K^p, \mathfrak{m}_i}^{\mathrm{cl}}$, such that 
\[\tilde{D}(1-X\cdot \mathrm{Frob}_v) = 1 - T_{1,v}X +T_{2,v}X^2 -\dots +(-1)^{2n+1}T_{n,v}X^{2n+1}.\]
\end{prop}
\begin{proof}
This follows from the existence of determinants with the required property valued in $\mathbb{T}_{K^pK_p,k, \mathfrak{m}_i}$ (constructed from the Galois representations associated to classical cuspforms on $G_0$), and $\mathbb{T}^{\mathrm{cl}}_{K^p, \mathfrak{m}_i}=\varprojlim_{k,K_p}\mathbb{T}_{K^pK_p,k, \mathfrak{m}_i}$.
\end{proof}
The proposition shows that $\mathbb{T}_{K^p, \mathfrak{m}_i}^{\mathrm{cl}}$ receives a surjection from a Galois pseudo-deformation ring, and hence it is complete local Noetherian.

Now let $\mathbb{T}_1:=\mathrm{Im}\left(\mathbb{T}_{K^p,\mathfrak{m}_i} \to \mathbb{T}_{K^p, \mathfrak{m}_i}^{\mathrm{cl}}\right)$. Then $\mathbb{T}_1$ is a dense local subring of $\mathbb{T}_{K^p,\mathfrak{m}_i}^{\mathrm{cl}}$. This shows that $\mathbb{T}_{K^p,\mathfrak{m}_i}^{\mathrm{cl}}$ is the completion of $\mathbb{T}_1$ with respect to the sequence of ideals $J_n=\mathfrak{m}_i^n\mathbb{T}_{K^p,\mathfrak{m}_i}^{\mathrm{cl}}\cap \mathbb{T}_1$.
\begin{lem} There exists a surjection $\mathbb{T}_{K^p,\mathfrak{m}_i}^{\mathrm{cl}}\twoheadrightarrow \tilde{\mathbb{T}}_{K^p,\mathfrak{m}_i}(m)$ respecting Hecke operators with the same name.
\end{lem}

\begin{proof}

Let $\mathbb{T}_2:=\mathrm{Im}\left(\mathbb{T}_{K^p,\mathfrak{m}_i} \to \tilde{\mathbb{T}}_{K^p,\mathfrak{m}_i}(m)\right)$. 
The proof of Theorem IV.3.1 of ~\cite{scholzetorsion} shows that  $\mathbb{T}_2$ is a quotient of  $\mathbb{T}_1$ (we remark that, up to localization at $\mathfrak{m}_i$, $\mathbb{T}_2$ is called $\mathbb{T}^{\mathrm{cl}}$ there). A Hecke operator in $J_n$ acts as 0 on $H^0(\mathfrak{X}^*_{K^pK_p}, (\omega^{int}_{K^pK_p})^{\otimes k}\otimes \mathfrak{I})_{\mathfrak{m}_i}/\pi^m$ (where $\mathfrak{X}$ and $\omega^{int}$ are as in loc.cit.) if $n$ is large compared to $k$, $K_p$, $m$. It then follows that $J_n$ acts act as 0 on $\oplus_{i=0}^{2d} H^i_c(X_{K^pK_p},\mathcal{O}/\pi^m)$ if $n$ is large compared to $K_p$ and $m$.
It follows that the the composition $\mathbb{T}_1\onto \mathbb{T}_2 \into \tilde{T}_{K^p}(m)_{\mathfrak{m}_i}$ is continous for the toplogy defined by $J_n$, and thus induces a map $\mathbb{T}_{K^p,\mathfrak{m}_i}^{\mathrm{cl}}\onto \tilde{\mathbb{T}}_{K^p,\mathfrak{m}_i}(m)$ which is necessarily a surjection.

\end{proof}

This fact together with Proposition \ref{prop: classical determinants} immediately gives the following, which is roughly equivalent to Corollary V.1.11 of ~\cite{scholzetorsion}. 
\begin{thm}\label{thm: det for completed cohomology} There is a continuous $(2n+1)$-dimensional determinant $\tilde{D}$ of $G_{F,S}$ with values in $\mathbb{\tilde{T}}_{K^p}(m)$, such that 
\[\tilde{D}(1-X\cdot \mathrm{Frob}_v) = 1 - T_{1,v}X +T_{2,v}X^2 -\dots +(-1)^{2n+1}T_{n,v}X^{2n+1}.\]
\end{thm}
\begin{cor}
 $\mathbb{\tilde{T}}_{K^p}(m)$ is a product of finitely many complete local Noetherian rings.
\end{cor}

\section{The boundary of the Borel-Serre compactification} \label{boundary}

In this section, we go back to studying the locally symmetric spaces for general linear groups. Assume that the level $K\subset GL_n(\mathbb{A}_{F,f})$ is neat, which can be achieved by increasing the level at $p$, as in the introduction. Let $X_K^\mathrm{BS}$ be the Borel-Serre compactification of $X_K=X_K^{\GL_n}$ (suppressing the field $F$ in the notation). This is a real manifold with corners, which is a compactification of $X_K$ such that the inclusion $X_K \hookrightarrow X_K^\mathrm{BS}$ is a homotopy equivalence. We have the excision long exact sequence for compactly supported cohomology: 
\[\dots \to H^i_c(X_K, \bar{\mathbb{F}}_p) \to H^i(X_K,\bar{\mathbb{F}}_p) \to H^i(X_K^\mathrm{BS}\setminus X_K,\bar{\mathbb{F}}_p)\to\dots\]
Theorem~\ref{thm: det for completed cohomology} allows us to understand Galois representations associated to systems of Hecke eigenvalues which contribute to the interior cohomology $H^i_!(X_K,\bar{\mathbb{F}}_p)$. To extend this to all systems of Hecke eigenvalues in $H^i(X_K,\bar{\mathbb{F}}_p)$, we also need to account for those systems of Hecke eigenvalues occurring in $H^i(X_K^\mathrm{BS}\setminus X_K,\bar{\mathbb{F}}_p)$. In this subsection, we recall the geometry of the Borel-Serre and reductive Borel-Serre compactifications and use induction to construct a determinant valued in (a quotient of) the Hecke algebra acting on $H^i(X_K^\mathrm{BS}\setminus X_K,\bar{\mathbb{F}}_p)$. We follow the original construction in ~\cite{borelserre} as well as the exposition in ~\cite{GHM}.

Note that the group $\mathrm{Res}_{F/\mathbb{Q}}GL_n$ has a non-trivial $\mathbb{Q}$-split torus in its center.  Let $H$ be the group 
\[(^0(\mathrm{Res}_{F/\mathbb{Q}}GL_n))^0:=\bigcap_\chi \ker \chi,\] where $\chi$ runs over all rationally defined characters of $\mathrm{Res}_{F/\mathbb{Q}}GL_n$.
The locally symmetric spaces $X_K$ can be identified with finitely many disjoint copies of generalized locally symmetric spaces for the algebraic group $H$ in the sense of ~\cite{GHM}. This identification follows from the definition in ~\cite{borelserre} and Paragraph 3.4 of ~\cite{GHM}. The boundary of the Borel-Serre compactification of $X_K$ has a stratification which runs over finitely many conjugacy classes of rational parabolic subgroups $P_H\subseteq H$. 

Every rational parabolic subgroup $P_H\subseteq H$ has a decomposition $P_H=L_{H}U_{H}$, where $L_{H}$ is the Levi quotient and $U_{H}$ is the unipotent radical of $P_H$. Let \[M_{H}=^0(L_{H}):=\bigcap_\chi \ker \chi^2,\] where $\chi$ runs over all rationally defined characters of $L_{H}$. The group of real points of every rational parabolic subgroup $P_H\subseteq H$ has a Langlands decomposition ${P_H}(\mathbb{R})=M_{H}(\mathbb{R})A_{H}(\mathbb{R})U_{H}(\mathbb{R})$. Now we can describe the stratification of the Borel-Serre compactification more precisely: the stratum corresponding to a parabolic subgroup $P_H$ can be identified with the locally symmetric space:
\[X^{P_H}:=P_H(\mathbb{Q})\backslash\left(P_H(\mathbb{A}_{f})/K^{P_H}_f \times P_H(\mathbb R)/A_H(\mathbb{R})K^{P_H}_\infty  \right)\] where $K^{P_H}_f:=K\cap P_H(\mathbb{A}_{f})\subset P_H(\mathbb{A}_{f})$ is a compact open subgroup and $K^{P_H}_\infty\subset M_{H}(\mathbb{R})$ is a maximal compact subgroup. This space is a nilmanifold bundle over the locally symmetric space associated to the Levi quotient $L_{H}$ of $P_H$
\[X^{L_H}:= L_H(\mathbb{Q})\backslash\left(L_H(\mathbb{A}_{f})/K^L_f \times L_H(\mathbb R)/A_H(\mathbb{R})K^{P_H}_\infty \right)\]
 where $K^L_f\subset L_H(\mathbb{A}_{f})$ is the image of $K^{P_H}_f$ under the projection $P_H(\mathbb{A}_{f}) \to L_H(\mathbb{A}_{f})$ and is a compact open subgroup. The fibers of this nilmanifold bundle are isomorphic to $(U_H(\mathbb{Q})\cap K^{P_H}_f)\backslash U_H(\mathbb{R})$, as in the proof of Lemma V.2.2 of ~\cite{scholzetorsion}. We note that $X^L_{K^L_f}$ can be identified with a locally symmetric space for the connected component of the identity in $M_H$, in the sense of (7.2.2) of ~\cite{GHM} (we can go to the connected component of the identity since $K^{P_H}_\infty\subset M_H(\mathbb{R})$ is a maximal compact subgroup). 
 
We now reinterpret $X^{L_H}$ as a product of generalized locally symmetric spaces. Each parabolic subgroup $P_H\subset H$ is the intersection with $H$ of a parabolic subgroup $P$ of $\mathrm{Res}_{F/\mathbb{Q}}GL_n$. Assume that $P$ has Levi quotient isomorphic to $L:=\prod_i \mathrm{Res}_{F/\mathbb{Q}}GL_{n_i}$, with $\sum_i n_i=n$. From the definition of the locally symmetric space $X^{L_H}$, we see that \[X^{L_H}\simeq L(\mathbb{Q})\backslash\left(L(\mathbb{A}_{f})/K^{L}_f \times L(\mathbb R)/A_{L}(\mathbb{R})K_{L} \right),\] where $K_{L}\subset {L}(\mathbb{R})$ is a maximal compact subgroup.

Moreover, the latter obviously decomposes as a product of locally symmetric spaces associated to the $\mathrm{Res}_{F/\mathbb{Q}}GL_{n_i}$. Therefore, we can compute the compactly supported cohomology of a stratum corresponding to $P$ using the Leray-Serre spectral sequence of a fibration and the Kunneth formula for compactly supported cohomology. 

For $\mathrm{Res}_{F/\mathbb{Q}}P$ running over $\mathbb{Q}$-conjugacy classes of maximal parabolic subgroups of $\mathrm{Res}_{F/\mathbb{Q}}GL_n$,  denote the corresponding stratum by $X^P_{K^P_f}$. We have an open immersion $X^{P}_{K^{P}_f}\hookrightarrow X^\mathrm{BS}_K\setminus X_K$, which leads to the long exact sequence for cohomology with compact support
\[\dots \to H^i_c(X^{P}_{K^{P}_f},\bar{\mathbb{F}}_p)\to H^i(X^\mathrm{BS}_K\setminus X_K,\bar{\mathbb{F}}_p ) \to  H^i(X^\mathrm{BS}_K\setminus (X_K \cap X^{P}_{K^{P}_f}), \bar{\mathbb{F}}_p) \to \dots.\] Thus $H^i(X^\mathrm{BS}_K\setminus X_K, \bar{\mathbb{F}}_p)$ is an extension between subquotients of $H^i_c(X^{P}_{K^{P}_f},\bar{\mathbb{F}}_p)$ and $H^i(X^\mathrm{BS}_K\setminus (X_K \cap X^{P}_{K^{P}_f}),\bar{\mathbb{F}}_p)$. After removing the strata corresponding to all maximal parabolic subgroups this way, we can continue this process for smaller parabolic subgroups.  The corresponding stratum will be open in what is left of the boundary and we can iterate the process above. We see that $H^i(X^\mathrm{BS}_K\setminus X_K, \bar{\mathbb{F}}_p)$ has a filtration whose graded pieces are subquotients of $H^i_c(X^{P}_{K^{P}_f},\bar{\mathbb{F}}_p)$ where $P$ runs through conjugacy classes of rational parabolic subgroups $\mathrm{Res}_{F/\mathbb{Q}}P$ of $\mathrm{Res}_{F/\mathbb{Q}}GL_n$. Furthermore, the length of the filtration depends only on $G$ and not on $K$.

Let $\pi : X^P_{K^P_f}\to X^L_{K^L_f}$ be the projection map. It is proper. The Leray spectral sequence tells us that \[E_2^{i,j}:=H^i_c(X^L_{K^L_f},R^j\pi_*\bar{\mathbb{F}}_p) \implies H^{i+j}_c(X^P_{K^P_f}, \bar{\mathbb{F}}_p).\] Therefore, $H^{i+j}_c(X^P_{K^P_f}, \bar{\mathbb{F}}_p)$ has a filtration whose graded pieces are subquotients of $H^i_c (X^L_{K^L_f}, R^j\pi_*\bar{\mathbb{F}}_p)$, and whose length is bounded in terms of $G$.

Let $\mathbb{T}^P_{F,S}:=\mathbb{Z}_p[P(\mathbb{A}_F^{S})//K^{P,S}]$ and $\mathbb{T}^L_{F,S}:=\mathbb{Z}_p[L(\mathbb{A}_F^{S})//K^{L,S}]$. We have maps $\mathbb{T}_{F,S}\to\mathbb{T}^P_{F,S}$ by restriction and $\mathbb{T}^P_{F,S}\to\mathbb{T}^L_{F,S}$ by integration along unipotent fibres. The composite is the unnormalized Satake transform. We obtain the following lemma, which is an analogue of Lemma V.2.3 of ~\cite{scholzetorsion}.
\begin{lemma}\label{lem: unnormalized Satake} The Leray spectral sequence  \[E_2^{i,j}=H^i_c(X^L_{K^L_f},R^j\pi_*\mathbb{\bar{F}}_p) \implies H^{i+j}_c(X^P_{K^P_f}, \mathbb{\bar{F}}_p)\] is equivariant for the action of $\mathbb{T}_{F,S}$ given by unnormalized Satake composed with the action of $\mathbb{T}^L_{F,S}$ on the $E_2$-page and by the natural action (via restriction to $\mathbb{T}^P_{F,S}$) on $H^{i+j}_c(X^P_{K^P_f}, \mathbb{\bar{F}}_p)$.
\end{lemma}
\begin{proof}
It suffices to check that the action of a Hecke corrsepondence $t\in \mathbb{T}_{G,S}$ on $X^P_{K^P_f}$ is compatible with the action of its unnormalized
Satake transform on $X^L_{K^L_f}$ via the natural projection $\pi:X^P_{K^P_f} \to X^L_{K^L_f}$. This statement can be checked at the level of points. Let $l\notin S$ and $t$ is the characteristic function of $G(\Z_l)hG(\Z_l)$.
Put $P(\Z_l)hP(\Z_l)=\coprod h_i P(\Z_l)$, so that $t$ acts as $[g]\mapsto \sum [gh_i]$ on $X^P_{K^P_f}$. Now if $P(\Z_l)hP(\Z_l)= \coprod u_i N(\Q_l)P(\Z_l)=\coprod u_i L(\Z_l)U(\Q_l)$, then $\pi(h_i)=\pi(h_j)$ iff $h_i$ and $h_j$ belong to the same right $L(Z_l)U(\Q_l)=U(\Q_l)P(\Z_l)$-coset. Thus the number of $h_i$ whose projection to $L$ are in the same right $L(\Z_l)$-coset is exactly given by integrating $t$ along $U(\Q_l)$. 
\end{proof}

 The sheaves $R^j\pi_*\bar{\mathbb{F}}_p$ encode the cohomology of the nilmanifold $(U_H(\mathbb{Q})\cap K^P_f)\backslash U_H(\mathbb{R})$. As in Lemma 1.9 of~\cite{HLTT}, we see that $R^j\pi_*\bar{\mathbb{F}}_p$ can be identified with the local system corresponding to the algebraic representation of $\mathrm{Res}_{F/\mathbb{Q}}L$ given by \[\wedge^j\left(\bigoplus_{\tau:F\hookrightarrow \mathbb{R}}\left(\bigoplus_{k<l} \mathrm{Std}_{n_k}\otimes \mathrm{Std}_{n_l}^{-1}\right)\right),\] where $n_k,n_l$ correspond to blocks in the Levi subgroup $L$ and $\mathrm{Std}_n$ is the standard representation of $GL_n(\cO_F)$ on an $(\bar{\mathbb{F}}_p)^n$. (The formula above comes from the fact that the action of $\mathrm{Res}_{F/\mathbb{Q}}L$ on $U_H(\mathbb{R})$ is the adjoint action.) If the level $K_p$ at $p$ is sufficiently small (depending only on $P$ and $L$), the local system $R^j\pi_*\mathbb{\bar{F}}_p$ is trivial. If this is the case, the Kunneth formula for compactly supported cohomology (which applies because the spaces we consider are manifolds) expresses the $E_2^{i,j}$ terms in terms of tensor products of $H^a_c(X^{\GL_{n_b}}_K, \bar{\mathbb{F}}_p)$. In general, we can always find a normal subgroup $K'$ of $K$ which is sufficiently small, and the Hochshild-Serre spectral sequence shows that $H^i_c(X^L_{K^L_f},R^j\pi_*\bar{\mathbb{F}}_p)$ has a filtration whose graded pieces are subquotients of (direct sums of) $H^a(K/K', H^b_c(X^L_{K'^L_f},\bar{\mathbb{F}}_p))$. The length of the filtration is bounded in terms of $G$, and the spectral sequence is equivariant with respect to $\mathbb{T}^L_{F,S}$.

The following summarizes the discussion in this section:
\begin{prop}\label{prop: cohomology of BS boundary}
 $H^i(X_K^\mathrm{BS}\setminus X_K,\bar{\mathbb{F}}_p)$ admits a filtration by $\mathbb{T}_{F,S}$-modules whose graded pieces are modules for the quotients of $\mathbb{T}_{F,S}$ acting on $H_c^i(X^L_{K_L}, \bar{\mathbb{F}}_p)$, where $\mathbb{T}_{F,S}$ acts via the unnormalized Satake transform $\mathbb{T}_{F,S}\to \mathbb{T}^L_{F,S}$ and $L$ is a rational Levi subgroup. The length of this filtration is bounded in terms of $G$ only. 
\end{prop}
\begin{rem}
The Proposition continues to hold (with exactly the same argument) if we replace cohomology with $\bar{\mathbb{F}}_p$-coefficient by cohomology with $\mathcal{O}/\pi^m$-coefficients, where $\mathcal{O}$ is a sufficiently large finite extension of $\Z_p$ (for example, its fraction field containing all the images of embeddings $F\hookrightarrow \bar{\Q}_p$ is large enough).
\end{rem}

\section{The image of complex conjugation}\label{complex conjugation}

The group $G_0=\mathrm{Res}_{F/\Q}\mathrm{Sp_{2n}}$ contains the group $\mathrm{Res}_{F/\mathbb{Q}}GL_n$ as a maximal Levi, therefore the interior cohomology of the locally symmetric space for  $G=\mathrm{Res}_{F/\mathbb{Q}}GL_n$ contributes to the boundary cohomology of the locally symmetric space for $G_0$. Therefore, one can use Theorem~\ref{thm: det for completed cohomology} to obtain a determinant valued in the Hecke algebra acting on the interior cohomology of the locally symmetric space for $G$. This is not yet the determinant for Galois representations associated to $G$, but rather involves a functorial transfer from $G$ to $G_0$.

More precisely, let $\mathbb{T}_{F,S}(K,i,m):=\mathrm{Im}(\mathbb{T}_{F,S}^{G}\to \mathrm{End}_{\mathcal{O}/\pi^m}(H^i_!(X_K^{\GL_n},\mathcal{O}/\pi^m)))$, where $\mathcal{O}$ is a finite extension of $\Z_p$ and $\pi$ a uniformizer. Let 
\[P_v(X):=1-q_v^{(n+1)/2}T^{GL_n}_{1,v}+\dots+(-1)^n q_v^{n(n+1)/2}T_{n,v}^{GL_n}X^n\]
and $P_v^\vee(X)$ be the polynomial with constant coefficient $1$ which is a scalar multiple of $P_v(1/X)$. The following is one of the main results in ~\cite{scholzetorsion}, proved in Corollary V.2.6 and Theorem V.3.1.

\begin{thm}\label{thm: det for interior cohomology}
There exists a nilpotent ideal $I\subset \mathbb{T}_{F,S}(K,i,m)$ with nilpotency index bounded only in terms of $G$ and are continuous $2n+1$- and $n$-dimensional determinant $\tilde D$, $D$ of $G_{F,S}$ valued in $\mathbb{T}_{F,S}(K,i,m)/I$ and such that
\[\tilde D(1-X\cdot \mathrm{Frob}_v) = (1-X)P_v(X)P_v^\vee(X)\]
\[ D(1-X\cdot \mathrm{Frob}_v) = P_v(X)\]
for all places $v\not\in S$. 
The same statement holds for the Hecke algebra acting on $H^i_c(X_K^{\GL_n},\mathcal{O}/\pi^m)$ and $H^i(X_K^{\GL_n},\mathcal{O}/\pi^m)$
\end{thm}

\begin{rem}\label{rem: interpolating c} 
Let $\mathbb{T}_{K,k}^{G_0}$ be the quotient of the Hecke algebra of  which acts on the space $H^0(X^{G_0*}_K,\omega_K^{\otimes k}\otimes\mathcal{I})$ of classical cuspforms for $G_0$ (see the notation in Section~\ref{completed cohomology}, which match with those in~\cite{scholzetorsion} section V.1). The determinant $\tilde D$ for interior cohomology is obtained by gluing determinants of the type constructed in Theorem~\ref{thm: det for completed cohomology}, which is glued from determinants valued in $\mathbb{T}_{K,k}^{G_0}$. Therefore, if all such determinants satisfy a certain identity, so will $\tilde D$. We will apply this observation to compute the coefficient of $X$ in $\tilde D(1-X\cdot c)$ and hence in $D(1-X\cdot c)$, where $c\in G_{F,S}$ is a choice of complex conjugation.  
\end{rem}
We can compute the characteristic polynomial of any complex conjugation in the determinants on $\mathbb{T}_{K,k}^{G_0}$ by the following:
\begin{lem}\label{lem: classical cusp form} Let $k>n$, $x\in \Spec{\mathbb{T}_{K,k}^{G_0}}(\mathbb{\bar Q}_p)$, let $\sigma_x:G_{F,S}\to GL_{2n+1}(\mathbb{Q}_p)$ be the Galois representation whose Frobenius eigenvalues match the Satake parameters determined by $x$ at places not in $S$. Then for any complex conjugation $c$
\[ \mathrm{tr}(\sigma_x)(c)=\pm 1\]
Thus the characteristic polynomial of $c$ is either $(1-X)^n(1+X)^{n+1}$ or $(1+X)^n(1-X)^{n+1}$.
\end{lem}
\begin{proof} By the proof of Corollary V.1.7 of ~\cite{scholzetorsion}, the cuspidal automorphic representation associated to $x$ determines cuspidal automorphic representations $\Pi_i$ of $GL_{n_i}$ for $i=1,\dots,m$ and integers $l_1,\dots,l_m$ such that
\begin{itemize}
\item $l_1n_1+\dots+l_mn_m = 2n+1$
\item each $\Pi_i$ is self-dual
\item each $\Pi_i|\cdot|^{2n+1-l_i}$ is regular L-algebraic
\item the infinitesimal characters associated to $\Pi_i|\cdot|^{(l_i-1)/2},\dots, \Pi_i,\dots,\Pi_i^{(1-l_i)/2}$ with $i=1,\dots, m$ form the multiset $\{k-1,\dots,k-n,0,n-k,\dots,1-k\}$.
\item the Galois representation associated to $x$ satisfies \[\sigma_x = \bigoplus_{i=1}^m \left(\sigma_i \oplus \sigma_i\chi_p^{-1} \oplus \dots \oplus \sigma_i\chi_p^{\otimes (1-l_i)}\right),\] where $\sigma_i$ is the Galois representation associated to the regular L-algebraic cuspidal automorphic representation $\Pi_i|\cdot|^{(1-l_i)/2}$ and $\chi_p^{-1}$ is the $p$-adic cyclotomic character (and also the Galois representation associated to the absolute value $|\cdot|$ by our normalization of class field theory).  
\end{itemize}
(This is where we make use of the results of~\cite{arthur}.)

We note that if $l_i$ is even, the trace of $c$ on $ \sigma_i \oplus \dots \sigma_i\chi_p^{\otimes (1-l_i)}$ is equal to $0$. If $l_i$ is odd, then $\Pi_i|\cdot|^{(l_i-1)/2)}$ is essentialy self-dual, with even multiplier character, so by [Ta], the trace of $c$ on $\sigma_i$ is $0$ if $n_i$ is even and $\pm 1$ if $n_i$ is odd. Therefore, the trace of $c$ on $ \sigma_i \oplus \dots \oplus \sigma^i\chi_p^{\otimes (1-l_i)}$ is $0$ if $n_i$ is even and $\pm 1$ if $n_i$ is odd.

We now show that there is at most one $i$ for which both $n_i$ and $l_i$ are odd. Indeed, $\Pi_i$ is conjugate self-dual, so the multiset of infinitesimal characters of $\Pi_i|\cdot|^{(l_i-1)/2}, \dots, \Pi_i,\dots, \Pi_i|\cdot|^{(1-l_i)/2}$ is symmetric about $0$. If both $n_i$ and $l_i$ are odd, this multiset contains an odd number of elements, so it must contain $0$. But from above, we see that $0$ can appear only once, so $n_il_i$ is odd for at most one $\Pi_i$. Since $2n+1$ is odd, we see that $n_il_i$ is odd for exactly one $\Pi_i$ and, therefore, $\mathrm{tr}(\sigma_x)(c)=\pm 1$.

\end{proof}

\begin{prop}\label{prop: interior cohomology} Let $\psi$ be a system of Hecke eigenvalues which occurs in the interior cohomology $H^i_!(X_K, \mathbb{\bar F}_p).$ Let $\sigma_\psi$ be the Galois representation associated to the determinant $D$ in Theorem~\ref{thm: det for interior cohomology} specialized to $\psi$. Then $\sigma_\psi(c)$  has $+1$ as an eigenvalue with multiplicity $\lceil \frac{n-1}{2} \rceil$ and $-1$ as an eigenvalue with multiplicity $\lfloor \frac{n+1}{2} \rfloor $.
\end{prop}
\begin{proof} It suffices to consider the case $p>2$, as the assertion is empty when $p=2$.
This follows from Lemma~\ref{lem: classical cusp form} and the construction of the determinant with the additional observation that the the determinant $\tilde{D}$ specialized at $\psi$ gives rise to the Galois representation $1\oplus \sigma_\psi \oplus (\sigma_\psi)^\vee$. This observation follows from the polynomial identity $\tilde D(1-X\cdot \mathrm{Frob}_v) = (1-X)P_v(X)P_v^\vee(X)$ appearing in Theorem ~\ref{thm: det for interior cohomology} and from the fact that $\sigma_\psi$ is the Galois representation associated to the specialization at $\psi$ of a determinant which matches $P_v(X)$ at almost all places. 

Now Remark~\ref{rem: interpolating c} and Lemma~\ref{lem: classical cusp form} tell us that, if $1$ is an eigenvalue of $\sigma_\psi(c)$ with multiplicity $a$ and $-1$ is an eigenvalue of $\sigma_\psi(c)$ with multiplicity $b$, then
\[(1-X)^{2a+1}(1+X)^{2b}=(1\pm X)^n(1\mp X)^{n+1}.\] 
Considering the cases $n$ even and $n$ odd separately gives the desired result.

\end{proof}

\begin{thm} \label{thm: mod p} Let $\psi$ be a system of Hecke eigenvalues which occurs in $H^{i}_c(X_{K}, \mathbb{\bar{F}}_p)$. Let $\sigma_\psi$ be the Galois representation associated to $\psi$. Then $\mathrm{tr}(\sigma_\psi)(c)$  is equal to $0$ if $n$ is even and to $-1$ if $n$ is odd. 
\end{thm}
\begin{proof} We prove this by induction on $n$. Note that the case $n=1$ is obvious, since the locally symmetric space in this case is compact so Proposition~\ref{prop: interior cohomology} applies. Assume that the theorem holds for all $n'< n$. If $\psi$ occurs in interior cohomology, we are done by Proposition~\ref{prop: interior cohomology}. If not, then $\psi$ occurs in the cohomology of the boundary of the Borel-Serre compactification of $X_K$. By Proposition ~\ref{prop: cohomology of BS boundary}, it is enough to show that for any Levi subgroup $L$ the determinant $D_L$ associated to $H^{i}_c(X^L_{K^L_f}, \mathbb{\bar{F}}_p)$ gives rise to Galois representations for which the trace of complex conjugation satisfies the conditions of the theorem.

Assume that $L=\prod_{i=1}^k GL_{n_i}(F)$. Then $D_L$ is obtained by taking the direct product of the determinants associated to the locally symmetric spaces for $GL_{n_i}(F)$ (where we mean the connected locally symmetric spaces), each of these appropriately twisted by powers of the cyclotomic character. The reason for the twists is that the action of $\mathbb{T}_{F,S}$ is compatible with the action of $\mathbb{T}^L_{F,S}$ via the \textit{unnormalized} Satake transform, as shown in Lemma~\ref{lem: unnormalized Satake}. Explicitly, for a place $v$ of $F$ not in $S$, the unnormalized Satake transform is the map \[\mathbb{Z}_p[T_1^{\pm 1},\dots T_n^{\pm 1}]^{S_n} \to \prod_{i=1}^k \mathbb{Z}_p[q_v^{1/2}][(T^i_1)^{\pm 1},\dots (T^i_{n_i})^{\pm 1}]^{S_{n_i}}, \] where \[T_j\mapsto q_v^{-(n_1+\dots+n_{i-1})/2+(n_{i+1}+\dots+n_k)/2}T^i_{j_i},\] with $i$ and $j_i$ uniquely determined by $1\leq j\leq n$. Since the normalized Satake transform is compatible with local Langlands up to $q_v^{(n_{(i)}+1)/2}$, we get that $D_L = \prod_{i=1}^k  D_i(\chi_p^{n_{i+1}+\dots+n_k})$, where $\chi_p$ is the $p$-adic cyclotomic character and $D_i$ is the determinant associated to the locally symmetric space for $GL_{n_i}(F)$. By the induction hypothesis, the characteristic polynomial of complex conjugation on each $D_i$ satisfies the conditions of the theorem. Noting that the cyclotomic character is odd, we see that the sign of the contribution from the successive $D_i$ switches every time an odd-dimensional determinant contributes to the sum. Therefore, $D_L$ also satisfies the conditions of the theorem.   

\end{proof}
The above theorem determines the conjugacy class of complex conjugation in $\sigma_\psi$ for $p$ odd. To prove the analogous statement for characteristic 0 system of Hecke eigenvalues, we need to have a version of the above theorem for cohomology with $\mathcal{O}/\pi^m$-coefficients.
\begin{prop} \label{prop: mod p^n}
Let $\psi:\mathbb{T}_{F,S}\to \mathcal{O}/\pi^m\mathcal{O}$ be a system of Hecke eigenvalue factoring through the Hecke algebra quotient acting on $H^i_c(X_K^{\GL_n},\mathcal{O}/\pi^m)$. There exists an integer $N_0$ depending only on $G$ such that the determinant $D$ in Theorem ~\ref{thm: det for interior cohomology} specializes to $\tilde{\psi}=\psi$ mod $\pi^{\floor*{m/N_0}-\mathrm{ord}_\pi(4)}$, and $D_{\tilde{\psi}}(1-X\cdot c)=1-\mathrm{tr}(c)X+\cdots$ with $\mathrm{tr}{c}=0$ if $n$ is even and $\mathrm{tr}(c)=-1$ if $n$ is odd. 

\end{prop}
\begin{proof}
The existence of the specialization to $\tilde{\psi}$ mod $\pi^{\floor*{m/N_0}}$ follows from the fact that the image $\psi(I)$ is nilpotent with nilpotency index bounded by that of the ideal $I$. Lemma ~\ref{lem: classical cusp form} shows that the identity 
\[(2\mathrm{tr}(\sigma_\psi)(c)+1)^2 =1\]
holds for the determinant on interior cohomology. By Proposition ~\ref{prop: cohomology of BS boundary} and an induction on $n$, the identity holds for the determinant on the Borel-Serre boundary of $X^{G}_K$, hence it holds for the determinant on $H^i_c(X_K^{G},\mathcal{O}/\pi^m)$.
Specializing via $\tilde{\psi}$ gives the identity $4\mathrm{tr}(c)(\mathrm{tr}(c)+1)=0$ in $\mathcal{O}/\pi^{\floor*{m/N_0}}$. By the same argument as in Theorem ~\ref{thm: mod p} either $\mathrm{tr}(c)$ or $\mathrm{tr}(c)+1$ is a unit, depending on the parity of $n$. 
\end{proof}
\begin{cor}
Let $\psi$ the system of Hecke eigenvalues of a regular algebraic cusp forms on $\GL_n/F$, with associated (p-adic) Galois representation $\sigma_\psi$. Then $\mathrm{tr}\sigma_\psi(c)\in\{0,\pm 1\}$. 
\end{cor}
\begin{proof}
It suffices to compute $\mathrm{tr}\sigma_\psi(c)$ in the case $\psi$ occurs in $H^i_c(X_K^{G},\mathcal{O})$, where we enlarge $\mathcal{O}$ to define $\sigma_\psi$ over it. Applying the above Proposition to $\psi$ mod $\pi^m$ gives $\mathrm{tr}\sigma_\psi(c)$ mod $\pi^{\floor*{m/N_0}-\mathrm{ord}_\pi(4)} \in \{0, -1\}$. (Since we're looking at $X_K^G$, we know that the trace lies in $\{0,-1\}$ rather than $\{0,\pm 1\}$.) Letting $m \to \infty$ gives the result.
\end{proof}
\begin{rem}
The corollary completely determines the conjugacy class of $\sigma_\psi(c)$. In general, for systems of Hecke eigenvalues mod $\pi^m$, one can refine Proposition ~\ref{prop: mod p^n} to compute the entire characteristic polynomial of $c$ modulo a smaller power of $\pi$.
\end{rem}
\bibliographystyle{amsalpha}
\bibliography{Complexconjugation} 
\end{document}